\documentclass{svjour3} 

\usepackage{mathrsfs}
\usepackage{amsmath}
\usepackage{amssymb}
\usepackage{txfonts}
\usepackage{amsfonts}
\usepackage{caption}
\usepackage{float}
\usepackage{comment}
\usepackage{url}

\numberwithin{equation}{section}
\numberwithin{definition}{section}
\numberwithin{theorem}{section}
\numberwithin{remark}{section}
\numberwithin{example}{section}

\newtheorem{LEMMA}[theorem]{Lemma}

\newtheorem{CONJECTURE}{Conjecture}

\numberwithin{CONJECTURE}{section}

\def\RR{\mathbb{R}}       
\def\CC{\mathbb{C}}       
\def\Rd{\mathbb{R}^d}     
\def\Rdzero{\mathbb{R}^d\setminus\{0\}}  
\def\vx{\boldsymbol{x}}   
\def\vy{\boldsymbol{y}}   
\def\vw{\boldsymbol{w}}   
\def\vxi{\boldsymbol{\xi}}
\def\NN{\mathbb{N}}       
\def\CC{\mathbb{C}}       

\def\ud{\mathrm{d}}       

\def\Schwartz{\mathcal{S}}        
\def\Schwartzm{\mathcal{S}_{2m}}  

\def\SI{\mathcal{SI}}  

\def\Continue{\mathrm{C}(\mathbb{R}^d)}           
\def\Continuezero{\mathrm{C}(\mathbb{R}^d\setminus\{0\})}     
\def\ContinueInf{\mathrm{C}^{\infty}(\mathbb{R}^d)}        
\def\Lloc{\mathrm{L}_1^{loc}(\mathbb{R}^d)} 
\def\Ltwo{\mathrm{L}_2(\mathbb{R}^d)}       
\def\Lone{\mathrm{L}_1(\mathbb{R}^d)}
\def\Linfty{\mathrm{L}_{\infty}(\mathbb{R}^d)}       

\def\NativePhi{\mathcal{N}_{\Phi}^m(\mathbb{R}^d)}  
\def\NativeG{\mathcal{N}_{G}^m(\mathbb{R}^d)}       

\newcommand{\norm}[1]{\left\lVert#1\right\rVert}  
\newcommand{\abs}[1]{\left\lvert#1\right\rvert}   

\newcommand{\Matlab}{{\sc Matlab}}

\def\pim{\pi_{m-1}(\mathbb{R}^d)}   

\def\vP{\mathbf{P}}            
\def\vQ{\mathbf{Q}}            
\def\vB{\mathbf{B}}
\def\vA{\mathbf{A}}

\def\FvP{\Hat{\mathbf{p}}}           

\def\Hp{\mathrm{H}_{\mathbf{P}}(\mathbb{R}^d)} 
\def\FT{\mathcal{FT}}

\def\Leb{\mathrm{L}}
\def\Hil{\mathcal{H}}
\def\Cont{\mathrm{C}}
\def\Native{\mathcal{N}}
\def\HP{\mathrm{H}_{\mathbf{P}}}

\def\Real{\mathrm{Re}}
\def\RealLloc{\mathrm{Re}(\mathrm{L}_{1}^{loc}(\mathbb{R}^d))}
\def\RealContinue{\mathrm{Re}(\mathrm{C}(\mathbb{R}^d))}   

\def\A{\mathcal{A}}

\def\Hilbert{\mathrm{H}}

\def\genFourG{\hat{\mathsf{g}}_m}

\begin{document}

\title{Reproducing Kernels of Generalized Sobolev
Spaces via a Green Function Approach with Distributional Operators}

\author{Gregory E. Fasshauer \and Qi Ye}

\institute{Gregory E. Fasshauer \at
              Department of Applied Mathematics, Illinois Institute of Technology, Chicago IL 60616 \\
              \email{fasshauer@iit.edu}
           \and
           Qi Ye \at
              Department of Applied Mathematics, Illinois Institute of Technology, Chicago IL 60616 \\
              Tel.: +1-312-567-5867\\
              Fax: +1-312-567-3135\\
              \email{qye3@iit.edu}
}

\date{}

\maketitle


\begin{abstract}
In this paper we introduce a generalized Sobolev space by defining a
semi-inner product formulated in terms of a vector distributional
operator $\vP$ consisting of finitely or countably many
distributional operators $P_n$, which are defined on the dual space
of the Schwartz space. The types of operators we consider
include not only differential operators, but also more general
distributional operators such as pseudo-differential operators. We
deduce that a certain appropriate full-space Green function $G$ with
respect to $L:=\vP^{\ast T}\vP$ now becomes a conditionally positive definite
function. In order to support this claim we ensure that the
distributional adjoint operator $\vP^{\ast}$ of $\vP$ is
well-defined in the distributional sense. Under sufficient
conditions, the native space (reproducing-kernel Hilbert space)
associated with the Green function $G$ can be isometrically embedded into or even
be isometrically equivalent to a generalized Sobolev space. As an application, we
take linear combinations of translates of the Green function with
possibly added polynomial terms and construct a multivariate
minimum-norm interpolant $s_{f,X}$ to data values sampled from an
unknown generalized Sobolev function $f$ at data sites located in
some set $X \subset \RR^d$. We provide several examples, such
as Mat\'ern kernels or Gaussian kernels, that illustrate how many
reproducing-kernel Hilbert spaces of well-known reproducing kernels
are isometrically equivalent to a generalized Sobolev space. These examples further illustrate how
we can rescale the Sobolev spaces by the vector distributional operator $\vP$.
Introducing the notion of scale as part of the definition of a generalized Sobolev space
may help us to choose the ``best'' kernel function for kernel-based approximation methods.

\end{abstract}



\section{Introduction}

A large and increasing number of recent books and research papers
apply radial basis functions or other kernel-based approximation
methods to such fields as scattered data approximation, statistical
or machine learning and the numerical solution of partial
differential equations, e.g.,
\cite{BerThAg04,Bou07,BouMeh03,Buh03,Fas07,KBU02ab,SchWen06,SchSmo02,Wah90,Wen05}.
Generally speaking, the fundamental underlying practical problem
common to many of these applications can be represented in the
following way. Given a set of data sites $X \subset \RR^d$ and
associated values $Y \subset \RR$ sampled from an unknown function
$f$, we use translates of a kernel function $\Phi$ and possible
polynomial terms to set up an interpolant $s_{f,X}$ to approximate
the function $f$. When $f$ belongs to the related native space of
$\Phi$, we can obtain error bounds and optimality properties of this
interpolation method. If $\Phi$ is only conditionally positive
definite (instead of the more straightforward positive definite
case), then it is known that the native space can also become a
reproducing-kernel Hilbert space with a reproducing kernel computed
from $\Phi$ along with additional polynomial terms (see
Section~\ref{s:CPD-NS} and \cite{Wen05}). Nevertheless, there still
remain a couple of difficult and challenging questions to be
answered for kernel methods: \emph{What kind of functions belong to
the related native space of a given kernel function, and which
kernel function is the best for us to utilize for a particular
application?} In particular, a better understanding of the native
space in relation to traditional smoothness spaces (such as Sobolev
spaces) is highly desirable. The latter question is partially
addressed by the use of techniques such as cross-validation and
maximum likelihood estimation to obtain optimally scaled kernels for
any particular application (see e.g.,~\cite{Ste99,Wah90}). However,
at the function space level, the question of scale is still in need
of a satisfactory answer. As we will illustrate shortly, the definition
of our generalized Sobolev spaces will include a notion of scale in a
rather natural way.

We
will deal with these questions in a different way than the authors of the survey paper
\cite{SchWen06} did. In this paper, we want to show that the kernel
functions and native spaces (reproducing kernels and
reproducing-kernel Hilbert spaces) can be computed via Green
functions and generalized Sobolev spaces induced by some vector
distributional operators $\vP:=\left(P_1,\cdots,P_n,\cdots\right)^T$
consisting of finitely or countably many distributional operators
$P_n$ (see Definition~\ref{d:DistrAdjoint}). We can further check
that differential operators are
special cases of these distributional operators.

Some well known examples covered by our theory include the Duchon
spaces and Beppo-Levi spaces associated with polyharmonic splines
(see Examples~\ref{e:spline} and~\ref{e:polyharmonic}). Moreover, in
\cite{Sch08} the author expressed a desire to choose the ``best''
scale parameter of a given kernel function for a particular
interpolation problem by looking at scaled versions of the classical
Sobolev space via different scale parameters.
Examples~\ref{ex:sobolev-spline}, \ref{e:Sobolevspline} and~\ref{e:Matern} tell us that we can
balance the role of different derivatives by selecting appropriate scale
parameters when reconstructing the classical Sobolev spaces by starting with
appropriately chosen inner products of for our generalized Sobolev spaces.
Finally, Example~\ref{e:Gaussian} shows that the native
space of the ubiquitous Gaussian function (the reproducing-kernel
Hilbert space of the Gaussian kernel) is isometrically equivalent to a generalized
Sobolev space, which can be applied to support vector machines and
in the study of motion coherence (see e.g.,~\cite{SSM98,YuiGrz88}).

In this article, we use the notation $\Real(\mathcal{E})$ to be
the collection of all real-valued functions of the function space
$\mathcal{E}$. For example, $\RealContinue$ denotes the collection
of all real-valued continuous functions on $\Rd$. $\SI$ is defined
as the collection of \emph{slowly increasing functions} which grow
at most like any particular fixed polynomial, i.e.,
\[
\SI:=\left\{ f:\Rd\rightarrow\CC:
f(\vx)=\mathcal{O}(\norm{\vx}_2^{m})\text{ as
}\norm{\vx_2}\to\infty\text{ for some }m\in\NN_0 \right\}.
\]
(The notation $f=\mathcal{O}(g)$ means that there is a positive
number $M$ such that $\abs{f}\leqslant M\abs{g}$.)
Roughly speaking, our generalized Sobolev space is a generalization
of the classical real-valued $\Leb_2(\Rd)$-based Sobolev space.
The real classical Sobolev space is usually given by
\[
\Hil^n(\Rd) := \left\{ f \in\RealLloc\cap\SI: D^{\alpha}
f\in\Ltwo\text{ for all }\abs{\alpha}\leq n,\alpha\in\NN^d_0\right\}
\]
with inner product
\[
(f,g)_{\Hil^n(\Rd)} := \sum_{\abs{\alpha}\leq n}\int_{\Rd}D^\alpha f(\vx)
\overline{D^\alpha g(\vx)}\ud\vx, \quad{ }f,g\in\Hil^n(\Rd),
\]

Our concept of a real \emph{generalized Sobolev space} (to
be defined in detail in Definition~\ref{d:GenSolSp} below) will be
of a very similar form, namely
\[
\Hp:= \Big\{ f\in\RealLloc\cap\SI:\{P_j
f\}_{j=1}^{\infty}\subseteq\Ltwo\text{ and }
\sum_{j=1}^{\infty}\norm{P_jf}_{\Ltwo}^2<\infty \Big\}
\]
with the semi-inner product
\[
(f,g)_{\Hp}:=\sum_{j=1}^{\infty}\int_{\Rd}P_j f(\vx)\overline{P_j g(\vx)}\ud\vx, \quad{
}f,g\in\Hp.
\]
Why do we use different vector distributional operators to set up the generalized Sobolev space? An important feature driving this definition is the fact that this will give us different semi-norms in which to measure the target function $f$ adding a notion of scale on top of the usual smoothness properties. As we discuss in Example~\ref{ex:sobolev-spline}, a scale parameter will control the semi-norm by affecting the weight of the various derivatives involved. This may guide us in finding the kernel function with ``optimal'' scale parameter to set up a kernel-based approximation for a given set of data values --- an important problem in practice for which no analytical solution exists.

Since the Dirac delta function $\delta_0$ at the origin is just a
tempered distribution belonging to the dual space of the Schwartz
space, the Green function $G$ we introduce in
Definition~\ref{d:GreenFct} needs to be regarded as a tempered
distribution as well. Thus we want to define a distributional
operator $L$ on the dual space of the Schwartz space so that
$LG=\delta_0$. The distributional operator and its distributional
adjoint operator are well-defined in Section~\ref{s:distributions}.
According to Theorem~\ref{t:GF-CPD}, we can prove that an even Green
function $G\in\RealContinue\cap\SI$ is a conditionally positive
definite function of some order $m\in\NN_0$. Therefore, we can
construct the related native space $\NativeG$ of $G$ as a complete
semi-inner product space. The native space can become a reproducing-kernel Hilbert space and its reproducing kernel is set up by the Green function and possible polynomial terms (see Section~\ref{s:CPD-NS}
and~\cite{Wen05}). Moreover, the distributional operator $L$ can be
computed by a vector distributional operator
$\vP:=\left(P_1,\cdots,P_n\right)^T$ and its distributional adjoint
$\vP^{\ast}$, i.e., $L=\vP^{\ast T}\vP=\sum_{j=1}^nP_j^{\ast}P_j$.
Under some sufficient conditions, we will further obtain a result in
Theorem~\ref{t:NS-Hp} that shows that the native space $\NativeG$ is
always a subspace of the generalized Sobolev space $\Hp$ and that
their semi-inner products are the same on $\NativeG$. This implies
that the usual native spaces can be isometrically embedded into our generalized Sobolev
spaces. By Lemma~\ref{l:Hpzero}, we know that
$\Hp\cap\Continue\cap\Ltwo$ is also a subspace of $\NativeG$.
Theorems~\ref{t:NS=Hp} and~\ref{t:NS=Hp0} tell us that $\NativeG$
may even be isometrically equivalent to $\Hp$. However, we provide
Example~\ref{e:Laplacian} to show that $\NativeG$ is not always
equivalent to $\Hp$. In other words, $\NativeG$ is sometimes just a
proper subspace of $\Hp$. We complete the proofs needed for the
theoretical framework in this article by applying generalized Fourier
transform (see Definition~\ref{d:GenFourier} and~\cite{Wen05}) and distributional Fourier transform (see
Definition~\ref{d:DistrFourier} and~\cite{SteWei75}) techniques.


\section{Background and Motivation}\label{s:background}

Given data sites $X=\{\vx_1,\cdots,\vx_N\}\subset\Rd$ (which we also
identify with the centers of our kernel functions below) and sampled
values $Y=\{y_1,\cdots,y_N\}\subset\RR$ of a real-valued continuous
function $f$ on $X$, we wish to approximate this function $f$ by a
linear combination of translates of a reproducing kernel $K$.

To this end we set up the interpolant in the form
\begin{equation}\label{interpolant}
s_{f,X}(\vx):=\sum_{j=1}^{N}c_jK(\vx,\vx_j),
\quad{}\vx\in\Rd,
\end{equation}
and require it to satisfy the additional interpolation conditions
\begin{equation}\label{interpolation_conditions}
s_{f,X}(\vx_j)=y_j,         \quad{  }j=1,\ldots,N.
\end{equation}
If $K$ is a \emph{positive definite}~\cite[Definition~6.24]{Wen05} reproducing kernel then
the above system
(\ref{interpolation_conditions}) is equivalent to a uniquely
solvable linear system
\[
\vA_{K,X}\boldsymbol{c}=\mathbf{Y},
\]
where
$\vA_{K,X}:=\left(K(\vx_j,\vx_k)\right)_{j,k=1}^{N,N}\in
\RR^{N\times N}$, $\boldsymbol{c}:=(c_1,\cdots,c_N)^T$ and
$\mathbf{Y}:=(y_1,\cdots,y_N)^T$.
Here a Hilbert
space $\Hilbert_K(\Rd)$ of functions $f:\Rd\rightarrow\RR$ is
called a \emph{reproducing-kernel Hilbert space} \cite[Definition~10.1]{Wen05} with a
\emph{reproducing kernel} $K:\Rd\times\Rd\rightarrow\RR$ if
\[
(1)~K(\cdot,\vy)\in\Hilbert_K(\Rd)\text{ and }
(2)~f(\vy)=(K(\cdot,\vy),f)_{\Hilbert_K(\Rd)},\text{ for all
}f\in\Hilbert_K(\Rd)\text{ and }\vy\in\Rd.
\]

It is well-known that the interpolant $s_{f,X}$ is the best approximation
of an unknown function $f\in\Hilbert_K(\Rd)$ fitting the sample values $Y$ on
the data sites $X$.

\begin{example}\label{ex:sobolev-spline}
We consider two reproducing kernels for differently scaled versions of the classical Sobolev space $\Hil^2(\RR)$: the kernel
\[
K^s(x,y):=\exp\left(-\frac{\sqrt{3}}{2}\abs{x-y}\right)\sin\left(\frac{1}{2}\abs{x-y}+\frac{\pi}{6}\right),\quad x,y\in\RR,
\]
and the Sobolev spline (Mat\'ern) kernel
\[
K(x,y):=\frac{1}{8\sigma^3}\left(1+\sigma\abs{x-y}\right)\exp\left(-\sigma\abs{x-y}\right),\quad x,y\in\RR,
\]
with scale parameter $\sigma>0$.
It is not difficult to show that these functions are Green functions of the differential operators $L^s:=I-\frac{\ud^2}{\ud x^2} + \frac{\ud^4}{\ud x^4}$ and $L:=(\sigma^2I-\frac{\ud^2}{\ud x^2})^2$, respectively. As a result the inner products for their real reproducing-kernel Hilbert spaces are
\[
(f,g)_{\Hilbert_{K^s}(\RR)}:=\int_{\RR}\left(f''(x)\overline{g''(x)}+f'(x)\overline{g'(x)}+f(x)\overline{g(x)}\right)\ud x,\quad f,g\in\Hil^2(\RR),
\]
and
\[
(f,g)_{\Hilbert_{K}(\RR)}:=\int_{\RR}\left(f''(x)\overline{g''(x)}+2\sigma^2f'(x)\overline{g'(x)}+\sigma^4f(x)\overline{g(x)}\right)\ud x,\quad f,g\in\Hil^2(\RR).
\]
We can also use the theoretical results of Section~\ref{s:GenSob-NS} to show that $\Hilbert_{K^s}(\RR)\equiv\Hil^2(\RR)\cong\Hilbert_{K}(\RR)$. This means that $\Hil^2(\RR)$ and $\Hilbert_K(\RR)$ are isomorphic and indicates that these reproducing-kernel Hilbert spaces are isometrically equivalent to generalized Sobolev spaces. More details are given in Example~\ref{e:Sobolevspline}.

This example shows that it may make sense to redefine the classical Sobolev space employing different inner products in terms of scale parameters even though $\Hil^2(\RR)$ and $\Hilbert_{K}(\RR)$ are composed of functions with the same smoothness properties and are not distinguished under standard Hilbert space theory (i.e., considered isomorphic). These different inner products provide us with a clearer understanding of the important role of the scale parameter.
This formulation allows us to think of $\sigma^{-1}$ as the natural length scale dependent on the weight of various derivatives. The choice of smoothness and scale now tell us which kernel to use for a particular application. This
choice may be performed by the user based on some a priori knowledge of the problem and based
directly on the data.

\end{example}

In the following section we briefly review how to use a conditionally positive definite function to construct reproducing kernels.


\section{Conditionally Positive Definite Functions and Native Spaces}\label{s:CPD-NS}

Most of the material presented in this section can be found in the excellent monograph \cite{Wen05}. For the reader's convenience we repeat here what is essential to our discussion later on.

\subsection{Conditionally Positive Definite Functions}\label{s:CPD}

\begin{definition}[{\cite[Definition~8.1]{Wen05}}]
A continuous even function $\Phi:\Rd\to\RR$ is said to be a
\emph{conditionally positive definite function of order $m\in\NN_0$}
if, for all $N\in\NN$, all pairwise distinct centers
$\vx_1,\ldots,\vx_N\in \Rd$, and all
$\boldsymbol{c}=(c_1,\cdots,c_N)^T\in\RR^N\setminus \{0\}$
satisfying
\[
\sum_{j=1}^{N}c_jp(\vx_j)=0
\]
for all $p\in\pim$, the quadratic form
\[
\sum_{j=1}^{N}\sum_{k=1}^{N}c_jc_k\Phi(\vx_j-\vx_k)>0.
\]
In the case $m=0$ with $\pi_{-1}(\Rd):=\{0\}$ the function $\Phi$ is called
\emph{positive definite}.
\end{definition}

In general, we can not hope for a continuous $\Phi$ to be $\Leb_1(\Rd)$-integrable so that it has a $\Leb_1(\Rd)$-Fourier transform.
However, $\Phi$ always has a generalized Fourier transform.

Next, we want to have a criterion to decide whether $\Phi$ is a
conditionally positive definite function of order $m\in\NN_0$. In
Wendland's book~\cite{Wen05}, the generalized
Fourier transform of order $m$ is employed to determine the conditional
positive definiteness of $\Phi$. Let a special test function space $\Schwartzm$~\cite[Definition~8.8]{Wen05} be defined as
\[
\Schwartzm:=\left\{\gamma\in\Schwartz:
\gamma(\vx)=\mathcal{O}\left(\norm{\vx}_2^{2m} \right) \text{ as }
\norm{\vx}_2\to 0\right\},
\]
where the \emph{Schwartz space $\Schwartz$}~\cite[Definition~5.17]{Wen05} consists of all functions
$\gamma\in\ContinueInf$ that satisfy
\[
\sup_{\vx\in\Rd}\abs{\vx^{\beta}D^{\alpha}\gamma(\vx)}\leqslant
C_{\alpha,\beta,\gamma}
\]
for all multi-indices $\alpha,\beta\in\NN_0^d$ with a constant
$C_{\alpha,\beta,\gamma}$.

\begin{definition}[{\cite[Definition~8.9]{Wen05}}]\label{d:GenFourier}
Suppose that $\Phi\in\Continue\cap\SI$. A measurable function
$\hat{\phi}\in\Leb_2^{loc}(\Rd\backslash\{0\})$ is called a
\emph{generalized Fourier transform} of $\Phi$ if there exists an
integer $m\in\NN_0$ such that
\[
\int_{\Rd}\Phi(\vx)\hat{\gamma}(\vx)\ud\vx=\int_{\Rd}\hat{\phi}(\vx)\gamma(\vx)\ud\vx,
\quad{}\text{for each }\gamma\in\Schwartzm.
\]
The integer $m$ is called the \emph{order} of $\hat{\phi}$.
\end{definition}
%
\begin{remark}
If $\Phi$ has a generalized Fourier transform of order $m$, then it
has also order $l\geqslant m$. If $\Phi\in\Ltwo\cap\Continue$, then
its $\Ltwo$-Fourier transform is a generalized Fourier transform of
any order.
\end{remark}
%
%
\begin{theorem}[{\cite[Theorem~8.12]{Wen05}}]\label{t:CPD-FT}
Suppose an even function $\Phi\in\RealContinue\cap\SI$ possesses a
generalized Fourier transform $\hat{\phi}$ of order $m$ which is
continuous on $\Rdzero$.
Then $\Phi$ is conditionally positive definite of order $m$ if and
only if $\hat{\phi}$ is nonnegative and nonvanishing.
\end{theorem}

\subsection{Native Space and Reproducing-Kernel Hilbert
Space}\label{s:NS-HS}

If $\Phi$ is conditionally positive definite of order $m$ then \cite[Chapter~10.3]{Wen05} shows that $\Phi$ can be used to create a reproducing kernel and its reproducing-kernel Hilbert space. We firstly set up a \emph{native space} $\NativePhi$~\cite[Definition~10.16]{Wen05}. It is a complete semi-inner product space and its null space of $\NativePhi$ is given by $\pim$, i.e., $\abs{p}_{\NativePhi}=0$ if
and only if $p\in\pim\subseteq\NativePhi$.  According to
\cite[Theorem~10.20]{Wen05}, $\NativePhi$ will become a
reproducing-kernel Hilbert space $\Hilbert_K(\Rd)$ with the new inner product
\[
(f,g)_{\Hilbert_K(\Rd)}:=(f,g)_{\NativePhi}+\sum_{k=1}^Q
f(\vxi_k)g(\vxi_k),\quad{}f,g\in\Hilbert_K(\Rd)=\NativePhi,
\]
and its reproducing kernel is given by
\[
\begin{split}
K(\vx,\vy):=&\Phi(\vx-\vy)-\sum_{k=1}^Q
q_k(\vx)\Phi(\vxi_k-\vy)-\sum_{l=1}^Q q_l(\vy)\Phi(\vx-\vxi_l)\\
&+\sum_{k=1}^Q\sum_{l=1}^Q
q_k(\vx)q_l(\vy)\Phi(\vxi_k-\vxi_l)+\sum_{k=1}^Q q_k(\vx)q_k(\vy),
\end{split}
\]
where $\{q_1,\cdots,q_Q\}$ is a Lagrange basis of $\pim$ with respect
to a $\pim$-unisolvent set $\{\vxi_1,\cdots,\vxi_Q\}\subset\Rd$ and $Q=\dim\pim$.
Moreover the reproducing kernel $K$ is positive definite by \cite[Theorem~12.9]{Wen05}. We can also check that the interpolation~(\ref{interpolant})-(\ref{interpolation_conditions}) by $K$
is equivalent to the interpolation by a linear combination of translates of the conditionally positive definite function $G$ of order $m$ along with a basis of polynomials $\pim$ (see~\cite[Chapter~8.5 and Chapter~12.3]{Wen05}), i.e.,
\[
s_{f,X}(\vx)=\sum_{j=1}^Nc_j\Phi(\vx-\vx_j)+\sum_{k=1}^Q\beta_kq_k(\vx),\quad \vx\in\Rd.
\]

\begin{theorem}[{\cite[Theorem~10.21]{Wen05}}]\label{t:NS-FT}
Suppose that $\Phi$ is a conditionally positive definite function of
order $m\in\NN_0$. Further suppose that $\Phi$ has a generalized
Fourier transform $\hat{\phi}$ of order $m$ which is continuous on
$\Rdzero$.
Then its native space is characterized by
\[
\begin{split}
\NativePhi=&\left\{ f\in\RealContinue\cap\SI:  f\text{ has a
generalized Fourier transform }\Hat{f}  \right. \\
&\left.\text{of order }m/2\text{ such that
}\hat{\phi}^{-1/2}\Hat{f}\in\Ltwo \right\},
\end{split}
\]
and its semi-inner product satisfies
\[
(f,g)_{\NativePhi}=(2\pi)^{-d/2}\int_{\Rd}\frac{\Hat{f}(\vx)\overline{\Hat{g}(\vx)}}{\hat{\phi}(\vx)}\ud\vx,
\quad{}f,g\in\NativePhi.
\]
\end{theorem}


\section{Green Functions and Generalized Sobolev Space Connected to Conditionally Positive Definite Functions and Native Space}\label{s:GenSob-NSG}

\subsection{Distributional Operators and Distributional Adjoint Operators}\label{s:distributions}

First, we can define a metric $\rho$ on the Schwartz space
$\Schwartz$ so that it becomes a Fr\'{e}chet space. Together with its metric $\rho$ the
Schwartz space $\Schwartz$ is regarded as the classical test
function space.

Let $\Schwartz'$ be the space of tempered distributions associated with $\Schwartz$ (the dual space of $\Schwartz$, or space of continuous linear functionals on $\Schwartz$). We introduce the notation
\[
\langle T,\gamma\rangle:=T(\gamma),\quad\text{for each
}T\in\Schwartz' \text{ and }\gamma\in\Schwartz.
\]
For each $f\in\Lloc\cap\SI$ there exists a
unique tempered distribution $T_f\in\Schwartz'$ such that
\[
\langle T_f,\gamma\rangle =\int_{\Rd}f(\vx)\gamma(\vx)\ud\vx,\quad\text{for each }
\gamma\in\Schwartz.
\]
So $f\in\Lloc\cap\SI$ can be viewed as an element of $\Schwartz'$
and we rewrite $T_f:=f$. This means that $\Lloc\cap\SI$ can be
isometrically embedded into $\Schwartz'$, i.e., $\Lloc\cap\SI\subseteq\Schwartz'$.
The \emph{Dirac delta function} (Dirac distribution) $\delta_0$
concentrated at the origin is also an element of $\Schwartz'$, i.e.,
$\langle \delta_0,\gamma\rangle =\gamma(0)$ for each
$\gamma\in\Schwartz$. Much more detail of the distributions are discussed in \cite[Chapter~7.1]{Hor04} and~\cite[Chapter~1.3]{SteWei75}.

Given a linear operator $P:\Schwartz'\rightarrow\Schwartz'$, is it
always possible to define a linear (adjoint) operator
$P^{\ast}:\Schwartz'\rightarrow\Schwartz'$ which also satisfies the usual adjoint
properties? The answer to this question is that it may not be possible for all $P$.
However, adjoint operators are well-defined for certain special
linear operators. We will refer to these special linear operators as
\emph{distributional operators} and to their adjoint operators as
\emph{distributional adjoint operators} in this article.

We first introduce these linear operators on $\Schwartz'$. Let
$\mathcal{P}^{\ast} : \Schwartz \to \Schwartz$ be a continuous linear operator.
Then a linear operator $P:\Schwartz'\rightarrow\Schwartz'$ induced by
$\mathcal{P}^{\ast}$ can be denoted via the form
\[
\langle PT,\gamma\rangle:=\langle
T,\mathcal{P}^{\ast}\gamma\rangle,\quad\text{for each
}T\in\Schwartz' \text{ and }\gamma\in\Schwartz,\quad{}\text{i.e.,
}P(T):=T\circ\mathcal{P}^{\ast}.
\]

Furthermore, if $P|_{\Schwartz}$ is a continuous operator from
$\Schwartz$ into $\Schwartz$, i.e.,
$\left\{P\gamma:\gamma\in\Schwartz\right\}\subseteq\Schwartz$ and
$\rho(P\gamma_n,P\gamma)\rightarrow0$ when
$\rho(\gamma_n,\gamma)\rightarrow0$, then we call the linear operator $P$
a \emph{distributional operator}.

Next we will show that the
adjoint operators of these distributional operators are well-defined
in the following way. In the same manner as before, we can denote another
linear operator $P^{\ast}:\Schwartz'\rightarrow\Schwartz'$ induced
by $P|_{\Schwartz}$, i.e.,
\[
\langle P^{\ast}T,\gamma\rangle:=\langle
T,P|_{\Schwartz}\gamma\rangle=\langle
T,P\gamma\rangle,\quad\text{for each }T\in\Schwartz' \text{ and
}\gamma\in\Schwartz.
\]

Fixing any $\tilde{\gamma}\in\Schwartz$, we have
\[
\langle P^{\ast}\tilde{\gamma},\gamma\rangle=\langle
\tilde{\gamma},P\gamma\rangle=\int_{\Rd} \tilde{\gamma}(\vx) P\gamma(\vx)\ud\vx=\langle
P\gamma, \tilde{\gamma} \rangle=\langle \gamma,
\mathcal{P}^{\ast}\tilde{\gamma} \rangle=\langle
\mathcal{P}^{\ast}\tilde{\gamma},\gamma \rangle,
\]
for each $\gamma\in\Schwartz$ which implies that
$P^{\ast}\tilde{\gamma}= \mathcal{P}^{\ast}\tilde{\gamma}$. Hence
$P^{\ast}|_{\Schwartz}=\mathcal{P}^{\ast}$ on $\Schwartz$ and
$P^{\ast}|_{\Schwartz}$ is a continuous operator from
$\Schwartz$ into $\Schwartz$. Therefore $P^{\ast}$ is also a
distributional operator. This motivates us to call $P^{\ast}$ the
\emph{distributional adjoint operator} of $P$. According to the
above definition, $P$ is also the distributional adjoint operator of
the distributional operator $P^{\ast}$.

%
\begin{remark}
In the standard literature~\cite[Chapter~8.3]{Hor04}
$P^{\ast}|_{\Schwartz}$ corresponds to the classical adjoint operator
of $P$. Here we can think of the classical adjoint operator
$P^{\ast}|_{\Schwartz}$ being extended to the distributional adjoint
operator $P^{\ast}$.
Our distributional adjoint operator differs from the adjoint
operator of a bounded linear operator defined in Hilbert space or
Banach space. Our operator is defined in the dual space of the
Schwartz space and it may not be a bounded operator if $\Schwartz'$
is defined as a metric space. But it is continuous when $\Schwartz'$
is given the weak-star topology as the dual of $\Schwartz$. However,
since the fundamental idea of our construction is similar to the
classical ones we also call this an adjoint.
\end{remark}

We now summarize the definitions of the distributional operator and its
adjoint operator.

\begin{definition}\label{d:DistrAdjoint}
Let $P,P^{\ast}:\Schwartz'\rightarrow\Schwartz'$ be two linear
operators. If $P|_{\Schwartz}$ and $P^{\ast}|_{\Schwartz}$ are
continuous operators from $\Schwartz$ into $\Schwartz$ such that
\[
\langle PT,\gamma\rangle=\langle T,P^{\ast}\gamma\rangle~\text{ and
}~\langle P^{\ast}T,\gamma\rangle=\langle T,P\gamma\rangle,
\quad\text{for each }T\in\Schwartz' \text{ and }\gamma\in\Schwartz,
\]
then $P$ and $P^{\ast}$ are said to be \emph{distributional
operators} and, moreover, $P^{\ast}$ (or $P$) is called a
\emph{distributional adjoint operator} of $P$ (or $P^{\ast}$).
\end{definition}

We will simplify the term distributional adjoint operator to adjoint operator in this article.

%
If $P=P^{\ast}$, then we call $P$ \emph{self-adjoint}. A distributional operator $P$
is called \emph{translation invariant} if
\[
\tau_hP\gamma=P\tau_h\gamma,\quad\text{for each }h\in\Rd\text{ and
}\gamma\in\Schwartz,
\]
where $\tau_h$ is defined by $\tau_h\gamma(\vx):=\gamma(\vx-h)$. A
distributional operator is called \emph{complex-adjoint
invariant} if
\[
\overline{P\gamma}=P\overline{\gamma},\quad\text{for each
}\gamma\in\Schwartz.
\]

Now we introduce two typical examples of distributional operators. One is the \emph{differential
operator} (with constant coefficients) which is a linear combination of the \emph{distributional derivatives}
$P:=D^{\alpha}:\Schwartz'\rightarrow\Schwartz'$. The distributional derivative is extended by the (strong) derivative
\[
D^{\alpha}:=\prod_{k=1}^{d}\frac{\partial^{\alpha_k}}{\partial
x_k^{\alpha_k}}, \quad{}\abs{\alpha}:=\sum_{k=1}^d\alpha_k,\quad{}
\alpha:=\left(\alpha_1,\cdots,\alpha_d\right)^T\in\NN_0^d,
\]
for the formula
\[
\langle D^{\alpha}T,\gamma\rangle:=(-1)^{\abs{\alpha}}\langle
T,D^{\alpha}\gamma\rangle, \quad\text{for each }T\in\Schwartz'
\text{ and }\gamma\in\Schwartz,
\]
(see~\cite[Chapter~1.5]{AdaFou03}). It is easy to check that the distributional derivative is a distributional operator. So we can determine that the differential
operator is a distributional operator, i.e.,
\[
P:=\sum_{\abs{\alpha}\leqslant
n}c_{\alpha}D^{\alpha},\quad P^{\ast}:=\sum_{\abs{\alpha}\leqslant
n}(-1)^{\abs{\alpha}}c_{\alpha}D^{\alpha},
\quad{}\text{where }c_{\alpha}\in\CC\text{
and }\alpha\in\NN_0^d,~n\in\NN_0.
\]

The other kind of distributional operator is defined for any fixed function
\[
\hat{p}\in\FT:=\left\{f\in\ContinueInf:D^{\alpha}f\in\SI\text{ for
each }\alpha\in\NN_0^d\right\}.
\]
It is obvious that all complex-valued polynomials belong to $\FT$.
Since $\hat{p}\gamma\in\Schwartz$ for each $\gamma\in\Schwartz$, we
can verify that the linear operator $\gamma\mapsto
\hat{p}\gamma$ is a continuous operator from $\Schwartz$ into
$\Schwartz$.
Thus this distributional operator $P$ related to
$\hat{p}$ is denoted as
\[
\langle PT,\gamma\rangle:=\langle
T,\hat{p}\gamma\rangle,\quad\text{for each }T\in\Schwartz' \text{
and }\gamma\in\Schwartz.
\]
We can further check that this operator is self-adjoint and
$Pg=\hat{p}g\in\Lloc\cap\SI$ if $g\in\Lloc\cap\SI$. Therefore we use the
notation $P:=\hat{p}$ for convenience. The $\FT$ space is also
applied in the definition of distributional Fourier transforms of distributional operators in
Section~\ref{s:DistFT}.

%
%

\subsection{Distributional Fourier Transforms}\label{s:DistFT}

We denote $\hat{\gamma}\in\Schwartz$ and $\check{\gamma}\in\Schwartz$ to be the $\Leb_1(\Rd)$-\emph{Fourier transform} and \emph{inverse} $\Leb_1(\Rd)$-\emph{Fourier transform} (unitary and using angular frequency) of the test function $\gamma\in\Schwartz$ as in \cite[Definition~5.15]{Wen05}.

Following the theoretical results of
\cite[Chapter~7.1]{Hor04} and~\cite[Chapter~1.3]{SteWei75} we can define the
\emph{distributional Fourier transform} $\Hat{T}\in\Schwartz'$ of
the tempered distribution $T\in\Schwartz'$ by
\[
\langle \Hat{T},\gamma\rangle :=\langle T,\Hat{\gamma}\rangle,
\quad\text{for each }\gamma\in\Schwartz.
\]
The fact $\langle T,\overline{\gamma}\rangle =\langle
\Hat{T},\overline{\Hat{\gamma}}\rangle$ implies that the $\Leb_1(\Rd)$-Fourier transform of $\gamma\in\Schwartz$ is the same as its
distributional transform. If $f\in\Ltwo$, then its $\Ltwo$-Fourier
transform is equal to its distributional Fourier transform.
The distributional Fourier transform $\Hat{\delta}_0$ of the Dirac delta
function $\delta_0$ is equal to $(2\pi)^{-d/2}$. Moreover, we can
check that the distributional Fourier transform map is an
isomorphism of the topological vector space $\Schwartz'$ onto
itself. This shows that the distributional Fourier transform map is
also a distributional operator.

If $\Phi\in\Continue\cap\SI$ has the generalized Fourier transform
$\hat{\phi}$ of order $m$, then its generalized Fourier transform and
its distributional Fourier transform coincide on the set
$\Schwartzm$, i.e.,
\[
\langle \Hat{\Phi},\gamma\rangle =\langle
\Phi,\hat{\gamma}\rangle=\int_{\Rd}\Phi(\vx)\hat{\gamma}(\vx)\ud\vx=\int_{\Rd}\hat{\phi}(\vx)\gamma(\vx)\ud\vx,
\quad\text{for each }\gamma\in\Schwartzm.
\]
Even if $\Phi$ does not have any generalized Fourier transform, it
always has a distributional Fourier transform $\hat{\Phi}$ since
$\Phi$ can be seen as a tempered distribution.

Our main goal in this subsection is to define the distributional Fourier transform of a
distributional operator induced by the $\FT$ space introduced in
Section~\ref{s:distributions}.
\begin{definition}\label{d:DistrFourier}
Let $P$ be a distributional operator. If there is a function
$\hat{p}\in\FT$ such that
\[
\langle \widehat{PT},\gamma\rangle  =\langle
\hat{p}\Hat{T},\gamma\rangle = \langle \Hat{T},\hat{p}\gamma\rangle,
\quad\text{for each }T\in\Schwartz' \text{ and }\gamma\in\Schwartz,
\]
then $\hat{p}$ is said to be a \emph{distributional Fourier
transform} of $P$.
\end{definition}
If $P$ has the distributional Fourier transform $\hat{p}$, then $P$
is translation-invariant because
$\widehat{\tau_hP\gamma}(\vx)=e^{-i\vx^Th}\hat{p}(\vx)\hat{\gamma}(\vx)=\widehat{P\tau_h\gamma}(\vx)$
for each $h\in\Rd$ and $\gamma\in\Schwartz$. Moreover, if $P$ is
complex-adjoint invariant and has the distributional Fourier
transform $\hat{p}$, then
\[
\langle \overline{\Hat{p}}\Hat{T},\overline{\gamma}\rangle =\langle
\Hat{T},\overline{\Hat{p}}\overline{\Hat{\check{\gamma}}}\rangle
=\langle \Hat{T},\overline{\widehat{P\check{\gamma}}}\rangle
=\langle T,\overline{P\check{\gamma}}\rangle=\langle
T,P\overline{\check{\gamma}}\rangle =\langle
P^{\ast}T,\overline{\check{\gamma}}\rangle = \langle
\widehat{P^{\ast}T},\overline{\gamma}\rangle
\]
for each $T\in\Schwartz'$ and $\gamma\in\Schwartz$. This shows that
$\overline{\Hat{p}}$ is the distributional Fourier transform of the
adjoint operator $P^{\ast}$ of $P$.

Because of $D^{\alpha}\hat{\gamma}=\left(\overline{\hat{p}}\gamma\right)^{\hat{}}$
for each $\gamma\in\Schwartz$,
we can show that any distributional derivative $D^{\alpha}$ has the
distributional Fourier transform $\hat{p}(\vx):=(i\vx)^{\alpha}$
where $i=\sqrt{-1}$.
This also implies that the
distributional Fourier transform $\hat{p}^{\ast}$ of its adjoint
operator $(-1)^{\abs{\alpha}}D^{\alpha}$ is equal to
$\hat{p}^{\ast}(\vx)=(-i\vx)^{\alpha}=\overline{\hat{p}(\vx)}$.
Furthermore, we can also obtain the distributional Fourier transform
of a differential operator in the same way, e.g.,
\[
\hat{p}(\vx)=\sum_{\abs{\alpha}\leqslant
n}c_{\alpha}(i\vx)^{\alpha},\quad{}\text{where
}P=\sum_{\abs{\alpha}\leqslant
n}c_{\alpha}D^{\alpha},~c_{\alpha}\in\CC\text{ and
}\alpha\in\NN_0^d,~n\in\NN_0.
\]

%
%

\subsection{Green Functions and Generalized Sobolev Space}\label{s:GenSob-NS}

\begin{definition}\label{d:GreenFct}
$G$ is the (full-space) \emph{Green function with respect to the
distributional operator $L$} if $G\in\Schwartz'$ satisfies the
equation
\begin{equation}\label{Green}
LG=\delta_0.
\end{equation}
\end{definition}
Equation~(\ref{Green}) is to be interpreted in the sense of distributions which means that
$\langle G,L^{\ast}\gamma \rangle=\langle LG,\gamma \rangle=\langle
\delta_0,\gamma \rangle=\gamma(0)$ for each $\gamma\in\Schwartz$.

According to Theorem~\ref{t:CPD-FT} and~\cite{MadNel90} we can
obtain the following theorem.

\begin{theorem}\label{t:GF-CPD}
Let $L$ be a distributional operator with
distributional Fourier transform $\hat{l}$. Suppose that $\Hat{l}$
is positive on $\Rdzero$. Further suppose that $\Hat{l}^{-1}\in\SI$
and that $\Hat{l}(\vx)=\Theta(\norm{\vx}_2^{2m})$ as
$\norm{\vx}_2\to 0$ for some $m\in\NN_0$. If the Green function
$G\in\RealContinue\cap\SI$ with respect to $L$ is an even function,
then $G$ is a conditionally positive definite function of order $m$
and
\[
\genFourG(\vx):=(2\pi)^{-d/2}\Hat{l}(\vx)^{-1},\quad{}\vx\in\Rd,
\]
is the generalized Fourier transform of order $m$ of $G$. (Here the
notation $f=\Theta(g)$ means that there are two positive numbers $M_1$
and $M_2$ such that $M_1\abs{g}\leqslant\abs{f}\leqslant
M_2\abs{g}$.)
\end{theorem}
%
\begin{proof}

First we want to prove that $\genFourG$ is the generalized
Fourier transform of order $m$ of $G$. Since $\Hat{l}^{-1}\in\SI$
and $\Hat{l}(\vx)=\Theta(\norm{\vx}_2^{2m})$ as $\norm{\vx}_2\to 0$
for some $m\in\NN_0$, the product $\genFourG\gamma$ is
integrable for each $\gamma\in\Schwartzm$. Let $\widehat{G}$ be the distributional Fourier
transform of $G$. If we can verify that
\[
\langle \widehat{G},\gamma\rangle=\int_{\Rd}\genFourG(\vx)\gamma(\vx)\ud\vx,\quad{}\text{for
each }\gamma\in\Schwartzm,
\]
then we are able to conclude that $\genFourG$ is the generalized
Fourier transform of $G$.

Since $\hat{l}$ is the distributional Fourier transform of the
distributional operator $L$ we know that $\hat{l}\in\FT$. Thus
$D^{\alpha}\left(\hat{l}^{-1}\right)\in\SI$ for each
$\alpha\in\NN_0^d$ because of $D^{\alpha}\hat{l}\in\SI$ and
$\hat{l}^{-1}\in\SI$. If $\hat{l}(0)>0$, then $\hat{l}^{-1}\in\FT$,
which implies that $\hat{l}^{-1}\gamma\in\Schwartz$ for each fixed
$\gamma\in\Schwartzm$. Hence
\[
\begin{split}
&\langle \widehat{G},\gamma\rangle=\langle
\hat{l}\widehat{G},\hat{l}^{-1}\gamma\rangle= \langle
\widehat{LG},\hat{l}^{-1}\gamma\rangle= \langle
\hat{\delta}_0,\hat{l}^{-1}\gamma\rangle
= \langle (2\pi)^{-d/2},\hat{l}^{-1}\gamma\rangle \\
= &
\int_{\Rd}(2\pi)^{-d/2}\hat{l}(\vx)^{-1}\gamma(\vx)\ud\vx=\int_{\Rd}\genFourG(\vx)\gamma(\vx)\ud\vx.
\end{split}
\]

If $\hat{l}(0)=0$, then $\hat{l}^{-1}$ does not belong to $\FT$. However,
since $\hat{l}\in\FT$ is positive on $\Rdzero$ we can find a
positive-valued sequence
$\{\hat{l}_{n}\}_{n=1}^{\infty}\subset\ContinueInf$ such that
\[
\hat{l}_{n}(\vx)=
\begin{cases}
\hat{l}(\vx),&~\norm{\vx}_2>n^{-1},\\
\hat{l}(\vx)+n^{-1},&~\norm{\vx}_2<n^{-2}.
\end{cases}
\]
In particular $l_1\equiv1$. And then
$\{\hat{l}_{n}\}_{n=1}^{\infty}\subset\FT$. It further follows that
$D^{\alpha}\hat{l}_{n}$ converges uniformly to $D^{\alpha}\hat{l}$
on $\Rd$ for each $\alpha\in\NN_0^d$.


We now fix an arbitrary $\gamma\in\Schwartzm$. Since $\hat{l}_{n}^{-1}\gamma$ and
$\hat{l}^{-1}\gamma$ have absolutely finite integral,
$\hat{l}_{n}^{-1}\gamma$ converges to $\hat{l}^{-1}\gamma$ in the
integral sense. Let $\gamma_{n}:=\hat{l}_{n}^{-1}\gamma$.
We can also check that $\left(\hat{l}\gamma_n\right)\hat{}$ converges to $\hat{\gamma}$ point wisely which indicates that $\int_{\Rd}G(\vx)
\left(\hat{l}\gamma_{n}\right)\hat{}(\vx)\ud\vx$ converges to $\int_{\Rd} G(\vx)\hat{\gamma}(\vx)\ud\vx$. Thus we have
\begin{align*}
&\langle
\gamma,\hat{G} \rangle_{\Schwartz}=
\lim_{n\to\infty}\langle
\hat{l}\gamma_{n},\hat{G} \rangle_{\Schwartz}=\lim_{n\to\infty}\langle
\gamma_{n},\widehat{LG} \rangle_{\Schwartz} = \lim_{n\to\infty}\langle
\gamma_{n},\hat{\delta}_0 \rangle_{\Schwartz} = \lim_{n\to\infty}\langle
\gamma_{n},(2\pi)^{-d/2} \rangle_{\Schwartz}\\
=&\lim_{n\to\infty}\int_{\Rd}(2\pi)^{-d/2}l_n(\vx)^{-1}\gamma(\vx)\ud\vx=
\int_{\Rd}(2\pi)^{-d/2}\hat{l}(\vx)^{-1}\gamma(\vx)\ud\vx
=\int_{\Rd}\hat{G}_m(\vx)\gamma(\vx)\ud\vx.
\end{align*}

Since $\genFourG\in\Continuezero$ is positive on $\Rdzero$
and $G\in\RealContinue\cap\SI$ is an even function, we can use
Theorem~\ref{t:CPD-FT} to conclude that $G$ is a conditionally
positive definite function of order $m$.
\qed
\end{proof}
%
\begin{remark}
If $L$ is a differential operator, then its distributional Fourier
transform $\hat{l}$ satisfies the conditions of
Theorem~\ref{t:GF-CPD} if and only if $\hat{l}$ has a polynomial
of the form $\hat{l}(\vx):=q(\vx)+a_{2m}\norm{\vx}_2^{2m}$, where
$a_{2m}>0$ and $q$ is a polynomial of degree greater
than $2m$ so that it is positive on $\Rdzero$, or $q\equiv0$.
\end{remark}

%
%

Now we can define the generalized Sobolev space induced by a vector
distributional operator $\vP=\left(P_1, \cdots ,P_n, \cdots
\right)^T$ similar as in~\cite[Definition~6]{Ye10}.
\begin{definition}\label{d:GenSolSp}
Consider the vector distributional operator  $\vP=\left(P_1, \cdots
,P_n, \cdots \right)^T$ consisting of countably many distributional
operators $\{P_j\}_{j=1}^{\infty}$. The real
\emph{generalized Sobolev space} induced by $\vP$ is defined by
\[
\Hp:= \Big\{ f\in\RealLloc\cap\SI: \{P_j
f\}_{j=1}^{\infty}\subseteq\Ltwo\text{ and }
\sum_{j=1}^{\infty}\norm{P_jf}_{\Ltwo}^2<\infty \Big\}
\]
and it is equipped with the semi-inner product
\[
(f,g)_{\Hp}:=\sum_{j=1}^{\infty}\int_{\Rd}P_j f(\vx)\overline{P_j g(\vx)}\ud\vx, \quad{
}f,g\in\Hp.
\]
\end{definition}

For example, if we let $P_j:=D^{\alpha}$ for each $\alpha\in\NN_0^d$ and $\abs{\alpha}\leqslant n$ and the others be zero operators, then the classical $\Leb_2$-based Sobolev space $\Hil^n(\Rd)\equiv\Hp$ is a special case of the generalized Sobolev space. If we choose the vector distributional operator $\vP$ as in Example~\ref{e:Sobolevspline} then
$\Hp$ and $\Hil^n(\Rd)$ are isomorphic to each other which indicates that we redefine the Sobolev space for different inner products using the scale parameter $\sigma>0$. Generalized Sobolev spaces can also become different kinds of Beppo-Levi spaces with corresponding semi-inner products (see Example~\ref{e:polyharmonic}). The reproducing-kernel Hilbert space of the Gaussian kernel will be isometrically equivalent to a generalized Sobolev space $\Hp$ as well as explained in Example~\ref{e:Gaussian}.

Now we discuss the relationship between the generalized Sobolev space and the native space.
In the following theorems of this section we only consider $\vP$
constructed by a finite number of distributional operators
$P_1,\ldots,P_n$ which means that $P_j:=0$ when $j>n$. If $\vP:=\left(P_1, \cdots ,P_n \right)^T$, then
the distributional operator
\[
L:=\vP^{\ast T}\vP=\sum_{j=1}^nP_j^{\ast}P_j
\]
is well-defined, where $\vP^{\ast}:=(P_1^{\ast}, \cdots ,P_n^{\ast}
)^T$ is the adjoint operator of $\vP$ as defined in Section~\ref{s:distributions}. If we suppose that $\vP$ is complex-adjoint
invariant with distributional Fourier transform $\FvP=(
\Hat{p}_1,\cdots,\Hat{p}_n )^T$, then the distributional Fourier
transform $\FvP^{\ast}=( \Hat{p}_1^{\ast}\cdots,\Hat{p}_n^{\ast})^T$
of its adjoint operator $\vP^{\ast}$ is equal to
$\overline{\FvP}=(\overline{\Hat{p}}_1,\cdots,\overline{\Hat{p}}_n)^T$.
Since
\[
\langle \widehat{P_j^{\ast}P_jT},\gamma \rangle=\langle
\hat{p}^{\ast}_j\widehat{P_jT},\gamma \rangle=\langle
\hat{p}_j\hat{T},\hat{p}^{\ast}_j\gamma \rangle=\langle
\overline{\hat{p}}_j\hat{p}_j\hat{T},\gamma \rangle=\langle
\abs{\hat{p}_j}^2\hat{T},\gamma \rangle
\]
for each $T\in\Schwartz'$ and $\gamma\in\Schwartz$, the
distributional Fourier transform $\hat{l}$ of $L$ is given by
\[
\hat{l}(\vx):=\sum_{j=1}^n\abs{\hat{p}_j(\vx)}^2=\norm{\FvP(\vx)}_2^2,\quad{}\vx\in\Rd.
\]
Moreover, since $\vP$ has a distributional Fourier transform,
$\vP$ is translation invariant (see Section~\ref{s:DistFT}).

We are now ready to state and prove our main theorem about the generalized Sobolev space $\Hp$ induced by a vector
distributional operator $\vP:=\left(P_1, \cdots ,P_n\right)^T$.

\begin{theorem}\label{t:NS-Hp}
Let $\vP:=\left(P_1, \cdots ,P_n\right)^T$ be a complex-adjoint invariant vector distributional operator with vector distributional Fourier transform $\FvP:=( \Hat{p}_1,\cdots,\Hat{p}_n)^T$ which is
nonzero on $\Rdzero$. Further
suppose that $\vx\mapsto\norm{\FvP(\vx)}_2^{-1}\in\SI$ and that
$\norm{\FvP(\vx)}_2=\Theta(\norm{\vx}_2^m)$ as $\norm{\vx}_2\to 0$
for some $m\in\NN_{0}$. If the Green function
$G\in\RealContinue\cap\SI$ with respect to $L=\vP^{\ast T}\vP$ is chosen so
that it is an even function,
then $G$ is a conditionally positive definite function of order $m$
and its native space $\NativeG$ is a subspace of the generalized
Sobolev space $\Hp$. Moreover, their semi-inner products are the
same on $\NativeG$, i.e.,
\[
(f,g)_{\NativeG}=(f,g)_{\Hp}, \quad{ }f,g\in\NativeG.
\]
\end{theorem}
%
\begin{proof}
By our earlier discussion the distributional Fourier transform
$\hat{l}$ of $L$ is equal to $\hat{l}(\vx)=\norm{\FvP(\vx)}_2^2$.
Thus $\hat{l}$ is positive on $\Rdzero$, $\Hat{l}^{-1}\in\SI$ and
$\Hat{l}(\vx)=\Theta(\norm{\vx}_2^{2m})$ as $\norm{\vx}_2\to 0$.
According to Theorem~\ref{t:GF-CPD}, $G$ is a conditionally
positive definite function of order $m$ and its generalized Fourier
transform of order $m$ is given by
\[
\genFourG(\vx):=(2\pi)^{-d/2}\Hat{l}(\vx)^{-1}=(2\pi)^{-d/2}\norm{\FvP(\vx)}_2^{-2},\quad{}\vx\in\Rd.
\]
With the material developed thus far we are able construct its
native space $\NativeG$ (see Section~\ref{s:NS-HS}).

Next, we fix any $f\in\NativeG$. According to Theorem~\ref{t:NS-FT},
$f\in\RealContinue\cap\SI$ possesses a generalized Fourier transform
$\Hat{f}$ of order $m/2$ and
$\vx\mapsto\Hat{f}(\vx)\norm{\FvP(\vx)}_2\in\Ltwo$. This means that
the functions $\Hat{p}_j\Hat{f}$ belong to $\Ltwo$, $j=1,\ldots,n$.
Hence we can define the function $f_{P_j} \in\Ltwo$ by
\[
f_{P_j}:=(\Hat{p}_j\Hat{f})\Check{}\in\Ltwo,\quad{}j=1,\ldots,n
\]
using the inverse $\Ltwo$-Fourier transform.

Since $\norm{\FvP(\vx)}_2=\Theta(\norm{\vx}_2^m)$ as
$\norm{\vx}_2\to 0$ we have
$\hat{p}_j(\vx)=\mathcal{O}(\norm{\vx}_2^m)$ as $\norm{\vx}_2\to 0$
for each $j=1,\ldots,n$. Thus
$\Hat{p}_j\check{\overline{\gamma}}\in\mathcal{S}_m$ for each
$\gamma\in\Schwartz$. Moreover, since
$\Hat{p}_j\check{\overline{\gamma}}=\Hat{p}_j\overline{\Hat{\gamma}}
=\overline{\Hat{p}_j^{\ast}\Hat{\gamma}}
=\overline{\widehat{P_j^{\ast}\gamma}}$ and the generalized and distributional Fourier transforms of $f$ coincide
on $\Schwartz_m$ we have
\[
\begin{split}
&\int_{\Rd}f_{P_j}(\vx)\overline{\gamma(\vx)}\ud\vx=\int_{\Rd}(\Hat{p}_j\Hat{f})\Check{}(\vx)\overline{\gamma(\vx)}\ud\vx
=\int_{\Rd}(\Hat{p}_j\Hat{f})(\vx)\check{\overline{\gamma}}(\vx)\ud\vx\\
=&\langle
\Hat{f},\Hat{p}_j\check{\overline{\gamma}} \rangle
=\langle \Hat{f},\overline{\widehat{P_j^{\ast}\gamma}}\rangle
=\langle f,\overline{P_j^{\ast}\gamma}\rangle  =\langle
f,P_j^{\ast}\overline{\gamma}\rangle =\langle
P_jf,\overline{\gamma}\rangle,\quad{}\gamma\in\Schwartz.
\end{split}
\]

This shows that $P_jf=f_{P_j}\in\Ltwo$. Therefore we know that
$f\in\Hp$.

To establish equality of the semi-inner products we let
$f,g\in\NativeG$. Then the Plancherel theorem~\cite{SteWei75} yields
\[
\begin{split}
(f,g)_{\Hp}&=\sum_{j=1}^{n}\int_{\Rd}f_{P_j}(\vx)\overline{g_{P_j}(\vx)}\ud\vx
=\sum_{j=1}^{n}\int_{\Rd}(\Hat{p}_j\Hat{f})(\vx)\overline{(\Hat{p}_j\Hat{g})(\vx)}\ud\vx\\
&=\int_{\Rd}\Hat{f}(\vx)\overline{\Hat{g}(\vx)}\norm{\FvP(\vx)}_2^2\ud\vx
=\int_{\Rd}\Hat{f}(\vx)\overline{\Hat{g}(\vx)}\Hat{l}(\vx)\ud\vx\\
&=(2\pi)^{-d/2}\int_{\Rd}\frac{\Hat{f}(\vx)\overline{\Hat{g}(\vx)}}{\genFourG(\vx)}\ud\vx
=(f,g)_{\NativeG}.
\end{split}
\]
\qed
\end{proof}
%
\begin{remark}
If each element of $\vP$ is just a differential operator then all their coefficients are real numbers because $\vP$ is complex-adjoint
invariant.
\end{remark}
%
The preceding theorem shows that $\NativeG$ can be isometrically embedded into
$\Hp$. Ideally, $\NativeG$ would be equal to $\Hp$, but this is not true in general.
However, if we impose some additional
conditions on $\Hp$, then we can obtain equality.
\begin{definition}\label{d:SchwartzDense}
Let $\vP:=\left(P_1,\cdots,P_n\right)^T$ be a vector distributional
operator. We say that the generalized Sobolev space $\Hp$ possesses
the \emph{$\Schwartz$-dense} property if for every $f\in\Hp$, every
compact subset $\Lambda\subset\Rd$ and every $\epsilon>0$, there
exists $\gamma\in\Schwartz\cap\Hp$ such that
\begin{equation}\label{Sdense}
\abs{f-\gamma}_{\Hp}<\epsilon \text{ and }
\norm{f-\gamma}_{\Leb_{\infty}(\Lambda)}<\epsilon,
\end{equation}
i.e., there is a sequence
$\left\{\gamma_n\right\}_{n=1}^{\infty}\subseteq\Schwartz\cap\Hp$
so that
\[
\abs{f-\gamma_n}_{\Hp}\rightarrow0 \text{ and }
\norm{f-\gamma_n}_{\Leb_{\infty}(\Lambda)}\rightarrow0,\text{ when
}n\rightarrow\infty.
\]
\end{definition}

Following the method of the proofs of \cite[Theorems~10.41 and
10.43]{Wen05}, we can complete the proofs of the following lemma and
theorem.
\begin{LEMMA}\label{l:flambda}
Let $\vP$ and $G$ satisfy the conditions of Theorem~\ref{t:NS-Hp}
and suppose that $\Hp$ has the $\Schwartz$-dense property. Assume we
are given the data sites $\{\vx_1,\cdots,\vx_N\}\subset\Rd$ and scalars
$\{\lambda_1,\cdots,\lambda_N \}\subset\RR$. If we define
$f_{\lambda}:=\sum_{k=1}^N\lambda_k G(\cdot-\vx_k)$, then for every
$f\in\Hp$ and every $\vx\in\Rd$ we have the representation
\begin{equation}\label{flambda}
(f_{\lambda}(\vx-\cdot),f)_{\Hp}=\sum_{k=1}^N\lambda_k f(\vx-\vx_k).
\end{equation}
\end{LEMMA}
%
\begin{proof}
Let us first assume that $\gamma\in\Schwartz\cap\Hp$. According to
Theorem~\ref{t:NS-Hp}, $f_{\lambda}\in\NativeG\subseteq\Hp$. Since
$\vP$ is translation invariant and complex-adjoint invariant we
have
\[
\begin{split}
&(f_{\lambda}(\vx-\cdot),\gamma)_{\Hp}=\sum_{j=1}^{n}\int_{\Rd}P_{j,\vy}f_{\lambda}(\vx-\vy)\overline{P_j\gamma(\vy)}\ud\vy
=\sum_{j=1}^{n}\int_{\Rd}P_{j,\vy}f_{\lambda}(\vx-\vy)P_j\gamma(\vy)\ud\vy\\
=&\sum_{j=1}^{n}\langle f_{\lambda}(\vx-\cdot),P_j^{\ast}P_j\gamma
\rangle = \int_{\Rd} f_{\lambda}(\vy)L_{\vy}\gamma(\vx-\vy)\ud\vy
=\sum_{k=1}^N\int_{\Rd}\lambda_k G(\vy-\vx_k)L_{\vy}\gamma(\vx-\vy)\ud\vy \\
=&\sum_{k=1}^N\lambda_k\langle LG,\gamma(\vx-\vx_k-\cdot)\rangle
=\sum_{k=1}^N\lambda_k\langle
\delta_0,\gamma(\vx-\vx_k-\cdot)\rangle
=\sum_{k=1}^N\lambda_k\gamma(\vx-\vx_k).
\end{split}
\]

For a general $f\in\Hp$ we fix $\vx\in\Rd$ and choose a compact set
$\Lambda\subset\Rd$ such that $\vx-\vx_k\in\Lambda$ for
$k=1,\ldots,N$. For any $\epsilon>0$, there is a
$\gamma\in\Schwartz\cap\Hp$ which satisfies Equation~(\ref{Sdense}).
Then two applications of the triangle inequality show that the
absolute value of the difference in the two sides of
Equation~(\ref{flambda}) can be bounded by
$\epsilon\left(\sum_{k=1}^N\abs{\lambda_k}+\abs{f_{\lambda}}_{\Hp}\right)$,
which tends to zero as $\epsilon\to 0$.
\qed
\end{proof}

%
\begin{theorem}\label{t:NS=Hp}
Let $\vP$ and $G$ satisfy the conditions of Theorem~\ref{t:NS-Hp}.
If $\Hp$ possesses the $\Schwartz$-dense property, then
\[
\NativeG\equiv\Hp.
\]
\end{theorem}
%
\begin{proof}
By Theorem~\ref{t:NS-Hp} we already know that $\NativeG$ is
contained in $\Hp$ and that their semi-inner products are the same
in the subspace $\NativeG$. Moreover, $\NativeG$ is a complete
subspace of $\Hp$. So, if we assume that $\NativeG$ were not the
whole space $\Hp$, then there would be an element $f\in\Hp$ which is
orthogonal to the native space $\NativeG$.

Let $Q=\dim\pi_{m-1}(\Rd)$ and $\{q_1,\cdots,q_Q\}$ be a Lagrange
basis of $\pi_{m-1}(\Rd)$ with respect to a
$\pi_{m-1}(\Rd)$-unisolvent subset
$\{\vxi_1,\cdots,\vxi_Q\}\subset\Rd$. We make the special choice of
the data sites $\{-\vx,-\vxi_1,\cdots,-\vxi_Q\}$ and scalars
$\left\{1,-q_1(\vx),\cdots,-q_Q(\vx)\right\}$ and correspondingly define
\[
f_{\lambda}:=G(\cdot+\vx)-\sum_{k=1}^{Q}q_k(\vx)G(\cdot+\vxi_k).
\]
Since $\Hp$ has the $\Schwartz$-dense property we can use
Lemma~\ref{l:flambda} to represent any $f\in\Hp$ in the form
\[
f(\vw+\vx)=\sum_{k=1}^Q
q_k(\vx)f(\vw+\xi_k)+(f_{\lambda}(\vw-\cdot),f)_{\Hp}.
\]
Since $G$ is even, we have $\vx\mapsto
f_{\lambda}(-\vx)\in\NativeG$. We now set $\vw=0$. The fact that $f$
is orthogonal to $\NativeG$ gives us
\[
f(\vx)=\sum_{k=1}^Q q_k(\vx)f(\xi_k)+(f_{\lambda}(-\cdot),f)_{\Hp}
=\sum_{k=1}^Q f(\xi_k)q_k(\vx).
\]
This shows that $f\in\pi_{m-1}(\Rd)\subseteq\NativeG$, and it
contradicts our first assumption. It follows that
$\NativeG\equiv\Hp$.
\qed
\end{proof}

%
\begin{LEMMA}\label{l:Hpzero}
Let $\vP$ and $G$ satisfy the conditions of Theorem~\ref{t:NS-Hp}.
Then
\[
\Hp\cap\Continue\cap\Ltwo\subseteq\NativeG.
\]
\end{LEMMA}
%
\begin{proof}

We fix any $f\in\Hp\cap\Continue\cap\Ltwo$ and suppose that $\hat{f}$ and
$\widehat{P_jf}$, respectively, are the $\Ltwo$-Fourier transforms of
$f$ and $P_jf$, $j=1,\ldots,n$. Using the Plancherel theorem
\cite{SteWei75} we obtain
\[
\int_{\Rd}(\Hat{p}_j\Hat{f})(\vx)\overline{(\Hat{p}_j\Hat{f})(\vx)}\ud\vx
=\int_{\Rd}\widehat{P_jf}(\vx)\overline{\widehat{P_jf}(\vx)}\ud\vx=\int_{\Rd}P_jf(\vx)\overline{P_jf(\vx)}\ud\vx<\infty.
\]
And therefore, with the help of the proof of Theorem~\ref{t:NS-Hp},
we have
\[
\begin{split}
\int_{\Rd}\frac{\abs{\Hat{f}(\vx)}^2}{\genFourG(\vx)}\ud\vx
&=(2\pi)^{d/2}\int_{\Rd}\abs{\Hat{f}(\vx)}^2\Hat{l}(\vx)\ud\vx
=(2\pi)^{d/2}\int_{\Rd}\abs{\Hat{f}(\vx)}^2\norm{\FvP(\vx)}_2^2\ud\vx
\\&=(2\pi)^{d/2}\sum_{j=1}^n\int_{\Rd}\abs{\Hat{f}(\vx)\Hat{p}_j(\vx)}^2\ud\vx
<\infty
\end{split}
\]
showing that $\genFourG^{-1/2}\Hat{f}\in\Ltwo$, where
$\genFourG$ is the generalized Fourier transform of $G$. And
now, according to Theorem~\ref{t:CPD-FT}, $f\in\NativeG$.
\qed
\end{proof}

This says that $\Hp\cap\Continue\cap\Ltwo$ can be isometrically embedded into
$\NativeG$. Moreover, according to Lemma~\ref{l:Hpzero} we can
immediately obtain the following theorem.

\begin{theorem}\label{t:NS=Hp0}
Let $\vP$ and $G$ satisfy the conditions of Theorem~\ref{t:NS-Hp}.
If $\Hp\subseteq\Ltwo$, then $G$ is positive definite and
\[
\NativeG\equiv\Hp.
\]
(It also indicates that $m=0$.)
\end{theorem}
\begin{proof}
Since $G\in\NativeG\subseteq\Hp\subseteq\Ltwo$, its generalized Fourier transform of any order is equal to its $\Ltwo$-Fourier transform which implies that $\genFourG\in\Ltwo\cap\Lone$. So $\vx\mapsto\norm{\FvP(\vx)}_2^{-1}\in\Ltwo$ and $\norm{\FvP(\vx)}_2=\Theta(1)$ as $\norm{\vx}_2\to 0$. According to Theorem~\ref{t:NS-Hp}, $G$ is a positive definite function.

We fix any $f\in\Hp\subseteq\Ltwo$. According to the proof of Lemma~\ref{l:Hpzero}, we have its distributional Fourier transform $\hat{f}\in\Leb_2(\Rd)$ and
\[
\norm{f}_{\Hp}^2=\sum_{j=1}^n\int_{\Rd}\abs{\widehat{P_jf}(\vx)}^2\ud\vx=\sum_{j=1}^n\int_{\Rd}\abs{\hat{p}_j(\vx)\hat{f}(\vx)}^2\ud\vx
=\int_{\Rd}\norm{\FvP(\vx)}_2^2\abs{\hat{f}(\vx)}^2\ud\vx.
\]
This means in particular that $\hat{f}\in\Leb_1(\Rd)$ because
\[
\int_{\Rd}\abs{\hat{f}(\vx)}\ud\vx\leq\left(\int_{\Rd}\norm{\FvP(\vx)}_2^2\abs{\hat{f}(\vx)}^2\right)^{1/2}
\left(\int_{\Rd}\norm{\FvP(\vx)}_2^{-2}\right)^{1/2}.
\]
Thus, the inverse $\Lone$-Fourier transform of $\hat{f}$ is equal to the inverse $\Ltwo$-Fourier transform of $\hat{f}$ which can be identified with $f$. This implies that $f\in\Cont(\Rd)$. According to Theorem~\ref{t:NS-Hp} and Lemma~\ref{l:Hpzero}, we have $\NativeG\equiv\Hp$.
\qed
\end{proof}
%
\begin{remark}
As Example~\ref{e:Laplacian} in Section~\ref{s:2D3D} shows, the native space $\NativeG$ will not always be
equal to the corresponding generalized Sobolev space $\Hp$.
\end{remark}
%
%


\section{Examples of Green Functions and their Related Generalized Sobolev Spaces}\label{s:examples}
\subsection{One-Dimensional Cases}

\begin{example}[Cubic Splines]\label{e:spline}
Consider the (scalar) distributional operator $\vP:=\ud^2/\ud x^2$ and
$L:=\vP^{\ast T}\vP=\ud^4/\ud x^4$. By integrating Equation
(\ref{Green}) four times we can obtain a family of possible Green
functions with respect to $L$, i.e.,
\[
G(x):=\frac{\abs{x}^3}{12}+a_3x^3+a_2x^2+a_1x+a_0,\quad{}x\in\RR,
\]
where $a_j\in\RR$, $j=0,1,2,3$. However, we want the Green function
to be an even function. Hence, we choose
\[
G(x):=\frac{1}{12}\abs{x}^3,\quad{}x\in\RR.
\]
This ensures that $\vP$ and $G$ satisfy the conditions of
Theorem~\ref{t:NS-Hp} and $\norm{\FvP(x)}_2=\abs{x}^2$. As a result,
the associated interpolant is given by
\[
s_{f,X}(x):=\sum_{j=1}^N\frac{c_j}{12}\abs{x-x_j}^3+\beta_2x+\beta_1,
\quad{}x\in\RR.
\]
This is the same as the cubic spline interpolant (see
\cite[Chapter~6.1.5]{BerThAg04}).

According to \cite[Theorem~10.40]{Wen05}, we can check
that $\HP(\RR)$
has the $\Schwartz$-dense property. Therefore, Theorem~\ref{t:NS=Hp}
tells us that $\Native_G^2(\RR)\equiv\HP(\RR)$ and it follows that
the cubic spline is the optimal interpolant for all
functions in the generalized Sobolev space $\HP(\RR)$.

\end{example}

\begin{example}[Tension Splines]\label{e:tensionsplines}
Let $\sigma>0$ be a tension parameter and consider the vector
distributional operator $\vP:=( \ud^2/\ud x^2,\sigma \ud/\ud x )^T$ and
$L:=\vP^{\ast T}\vP=\ud^4/\ud x^4-\sigma^2 \ud^2/\ud x^2$. Then
\[
G(x):=-\frac{1}{2\sigma^3}\left(\exp(-\sigma\abs{x})+\sigma\abs{x}\right),
\quad{}x\in\RR,
\]
is a solution of Equation~(\ref{Green}). We can verify that $\vP$
and $G$ satisfy the conditions of Theorem~\ref{t:NS-Hp} and that
$\norm{\FvP(x)}_2=(\abs{x}^4+\sigma^2\abs{x}^2)^{1/2}=\Theta(\abs{x})$
as $\abs{x}\to 0$. So $G$ is a conditionally positive definite function of
order $1$. This yields the same interpolant as the tension spline
interpolant~\cite{BouMeh03,Sch66}.

According to the Sobolev
inequality~\cite{AdaFou03} and \cite[Theorem~10.40]{Wen05},
$\HP(\RR)$ has the $\Schwartz$-dense property which implies that
$\Native^1_G(\RR)\equiv\HP(\RR)$. Theorem~\ref{t:NS=Hp} and
\cite[Theorem~13.2]{Wen05} provide us with the same optimality
property as stated in~\cite{BouMeh03,Sch66}.

\end{example}

\begin{example}[Example~\ref{ex:sobolev-spline}]\label{e:Sobolevspline}

We use the theoretical results of Section~\ref{s:GenSob-NS} to verify the reproducing-kernel properties of Example~\ref{ex:sobolev-spline}.
Here we only give details for the Sobolev spline kernel as the other kernel can be treated in the same way.
Let $\vP:=( \ud^2/\ud x^2, \sqrt{2}\sigma \ud/\ud
x,\sigma^2 I )^T$ and $L:=\vP^{\ast T}\vP=(\sigma^2I-\ud^2/\ud x^2)^2$.
It is known that the Green function with respect to $L$ is the Mat\'ern function
\[
G(x):=\frac{1}{8\sigma^3}(1+\sigma\abs{x})\exp(-\sigma\abs{x}),\quad{}x\in\RR.
\]
Since $\vP$ and $G$ satisfy the conditions of Theorem~\ref{t:NS-Hp}
and $\norm{\FvP(x)}_2= \sigma^2+x^2 =\Theta(1)$
as $\abs{x}\to 0$, we can determine that $G$ is positive definite. It is easy to check that $\HP(\RR)\cong\Hil^2(\RR)$.
Since we have $\Hil^2(\RR)\subseteq\Leb_2(\RR)$, we can use Theorem~\ref{t:NS=Hp0} to check that $\Native^0_G(\RR)\equiv\HP(\RR)$. As discussed in Section~\ref{s:NS-HS},
the reproducing kernel and its reproducing-kernel Hilbert space have the forms $K(x,y)=G(x-y)$ and $\Hilbert_K(\RR)\equiv\HP(\RR)$.

\end{example}

%
%

\subsection{Two-Dimensional Cases}\label{s:2D3D}

\begin{example}[Thin Plate Splines]\label{e:TPS}
Let $\vP:=(\partial^2/\partial
x_1^2,\sqrt{2}\partial^2/\partial x_1\partial
x_2,\partial^2/\partial x_2^2)^T$ so that $L:=\vP^{\ast
T}\vP=\Delta^2$. It is well-known that the fundamental solution of
the Poisson equation on $\RR^2$ is given by $\vx \mapsto
\log\norm{\vx}_2$, i.e., $\Delta\log\norm{\vx}_2=-2\pi\delta$.
Therefore Equation~(\ref{Green}) is solved by
\begin{equation}\label{TPS}
G(\vx):=\frac{1}{8\pi}\norm{\vx}_2^2\log\norm{\vx}_2,\quad{}\vx\in\RR^2.
\end{equation}
Since $\vP$ and $G$ satisfy the conditions of Theorem~\ref{t:NS-Hp}
and $\norm{\FvP(\vx)}_2=\norm{\vx}_2^2$, $G$ is a conditionally
positive definite function of order $2$ and its related interpolant has the
form
\begin{equation}\label{TPSinterpolant}
s_{f,X}(\vx):=\sum_{j=1}^N c_j G(\vx-\vx_j)+\beta_3 x_2+\beta_2
x_1+\beta_1,\quad{}\vx=(x_1,x_2)\in\RR^2.
\end{equation}
Moreover, according to \cite[Theorem~10.40]{Wen05}, we can verify
that $\HP(\RR^2)$ has the $\Schwartz$-dense property. Therefore,
$\Native_G^2(\RR^2)\equiv\HP(\RR^2)$ by Theorem~\ref{t:NS=Hp}.
Equation (\ref{TPSinterpolant}) is known as the \emph{thin plate
spline} interpolant (see~\cite{Bou07,Duc77,KBU02ab}).

Finally, we consider the Duchon semi-norm mentioned in \cite{Duc77},
i.e.,
\[
\abs{f}_{D_2}^2:=\int_{\RR^2}\abs{ \frac{\partial^2f(\vx)}{\partial
x_1^2} }^2+2\abs{ \frac{\partial^2f(\vx)}{\partial x_1
\partial x_2} }^2+\abs{
\frac{\partial^2f(\vx)}{\partial x_2^2}
}^2\ud\vx,\quad{}f\in\Leb_1^{loc}(\RR^2)\cap\SI,
\]
and the Duchon semi-norm space
\[
\Hil_{D_2}(\RR^2):=\left\{
f\in\Real(\Leb_1^{loc}(\RR^2))\cap\SI:\abs{f}_{D_2}<\infty \right\}.
\]
If we define $\vP$ as above, then it is easy to check that
$\HP(\RR^2)\equiv\Hil_{D_2}(\RR^2)$. According to
\cite[Theorems~13.1~and~13.2]{Wen05} we can conclude that the Duchon
semi-norm space possesses the same optimality properties as those
listed in \cite{Duc77}.
\end{example}
%

The following example shows that the same Green function $G$ can
generate different generalized Sobolev spaces $\Hp$. Moreover, it
illustrates the fact that the native space $\NativeG$ may be a
proper subspace of $\Hp$.
\begin{example}[Modified Thin Plate Splines]\label{e:Laplacian}
Let $\vP:=\Delta$ and $L:=\vP^{\ast T}\vP=\Delta^2$. We find that
the thin plate spline (\ref{TPS}) is also the Green function with
respect to the operator $L$ defined here. The associated interpolant
is again of the form (\ref{TPSinterpolant}).

We now consider the Laplacian semi-norm
\[
\abs{f}_{\Delta}^2:=\int_{\RR^2}\abs{\Delta
f(\vx)}^2\ud\vx,\quad{}f\in\Leb_1^{loc}(\RR^2)\cap\SI,
\]
and the Laplacian semi-norm space
\[
\Hil_{\Delta}(\RR^2):=\left\{
f\in\Real(\Leb_1^{loc}(\RR^2))\cap\SI:\abs{f}_{\Delta}<\infty
\right\}.
\]

It is easy to verify that $\HP(\RR^2)\equiv\Hil_{\Delta}(\RR^2)$.
However, it is known that $\Hil_{D_2}(\RR^2)$ is a proper subspace
of $\Hil_{\Delta}(\RR^2)$ since $q\in\Hil_{\Delta}(\RR^2)$ but
$q\not\in\Hil_{D_2}$ where $q(\vx):=x_1x_2$. Therefore, due to
Example~\ref{e:TPS}, we conclude that
\[
\Native_G^2(\RR^2)\equiv\Hil_{D_2}(\RR^2)\varsubsetneqq\Hil_{\Delta}(\RR^2)\equiv\HP(\RR^2).
\]

Instead of working with the polynomial space $\pi_1(\RR^2)$ which is
used to define $\Native_G^2(\RR^2)$, we can construct a new native
space $\Native_G^{\mathscr{P}}(\RR^2)$ for $G$ by using another
finite-dimensional space $\mathscr{P}$ of
$\Real(\Cont^2(\RR^2))\cap\SI$ such that
$\Native_G^{\mathscr{P}}(\RR^2)$ may be equal to the other subspace
of $\HP(\RR^2)$. First we can verify that the finite-dimensional
space $\mathscr{P}:=\textrm{span} \left\{ \pi_1(\RR^2)\cup\{ q \}
\right\}$ is a subspace of the null space of $\HP(\RR^2)$. Since
$\pi_1(\RR^2)\subset\mathscr{P}$ and $G$ is a conditionally positive
definite function of order 2, we know that $G$ is also conditionally
positive definite with respect to $\mathscr{P}$. Hence, the new
native space $\Native_G^{\mathscr{P}}(\RR^2)$ with respect to $G$
and $\mathscr{P}$ is well-defined (see \cite[Chapter~10.3]{Wen05}).
We can further check that $\Native_G^{\mathscr{P}}(\RR^2)$ is a
subspace of $\HP(\RR^2)$ but it is larger than $\Native_G^2(\RR^2)$,
i.e.,
$\Native_G^2(\RR^2)\subsetneqq\Native_G^{\mathscr{P}}(\RR^2)\subseteq\HP(\RR^2)$.

So we can obtain a modification of the thin plate spline interpolant
based on $\mathscr{P}$:
\[
s^{\mathscr{P}}_{f,X}(\vx):=\sum_{j=1}^N c_j G(\vx-\vx_j)+\beta_4
x_1 x_2+\beta_3 x_2+\beta_2
x_1+\beta_1,\quad{}\vx=(x_1,x_2)\in\RR^2.
\]

\end{example}

\begin{CONJECTURE}
Motivated by Example~\ref{e:Laplacian} we audaciously \emph{guess}
the following extension of the theorems in
Section~\ref{s:GenSob-NS}: Let $\mathbf{P}$ and $G$ satisfy the
conditions of Theorem~\ref{t:NS-Hp}. If the subspace $\mathscr{P}$
of the null space of $\Hp$ is a finite-dimensional subspace and
$\pim\subseteq\mathscr{P}$, then the new native space
$\Native_G^{\mathscr{P}}(\RR^2)$ with respect to $G$ and
$\mathscr{P}$ is a subspace of $\Hp$.
\end{CONJECTURE}

%
%

\subsection{d-Dimensional Cases}
\begin{example}[Polyharmonic Splines]\label{e:polyharmonic}
This is a generalization of the earlier Examples~\ref{e:spline} and
\ref{e:TPS}. Let $\vP:=(\partial^m/\partial x_1^m, \cdots,
\left(m!/\alpha!\right)^{1/2}D^{\alpha}, \cdots,
\partial^m/\partial x_d^m )^T$ consisting of all
$\left(m!/\alpha!\right)^{1/2}D^{\alpha}$ with
$\abs{\alpha}=m>d/2$. We further denote $L:=\vP^{\ast
T}\vP=(-1)^m\Delta^m$. Then the \emph{polyharmonic spline on $\Rd$}
is the solution of Equation~(\ref{Green}) (see
\cite[Chapter~6.1.5]{BerThAg04}), i.e.,
\[
G(\vx):=
\begin{cases}
\frac{\Gamma(d/2-m)}{2^{2m}\pi^{d/2}(m-1)!}\norm{\vx}_2^{2m-d} & \textrm{for $d$ odd,}\\
\frac{(-1)^{m+d/2-1}}{2^{2m-1}\pi^{d/2}(m-1)!(m-d/2)!}\norm{\vx}_2^{2m-d}\log\norm{\vx}_2
& \text{for $d$ even.}
\end{cases}
\]
We can also check that $\vP$ and $G$ satisfy the conditions of
Theorem~\ref{t:NS-Hp} and that $\norm{\FvP(\vx)}_2=\norm{\vx}_2^m$.
Therefore $G$ is a conditionally positive definite function of order $m$.
Furthermore, according to \cite[Theorem~10.40]{Wen05}, we can verify
that $\Hp$ has the $\Schwartz$-dense property. Therefore,
$\NativeG\equiv\Hp$ by Theorem~\ref{t:NS=Hp}.

We now consider the Beppo-Levi space of order $m$ on $\Rd$, i.e.,
\[
BL_m(\Rd):=\left\{ f\in\RealLloc: D^{\alpha}f\in\Ltwo\text{ for all
}\abs{\alpha}=m \right\}
\]
equipped with the semi-inner product
\[
(f,g)_{BL_m(\Rd)}:=\sum_{\abs{\alpha}=m}\frac{m!}{\alpha!}\int_{\Rd}D^{\alpha}f(\vx)\overline{D^{\alpha}g(\vx)}\ud\vx,
\quad{}f,g\in BL_m(\Rd).
\]
According to~\cite{LigWay99}, we
know that $BL_m(\Rd)\subseteq\RealLloc\cap\SI$ whenever $m>d/2$.
Hence $\Hp\equiv BL_m(\Rd)$.

By the way, it is well-known that $G$ is also conditionally
positive definite of order $l:=m-\lceil d/2\rceil+1$
(see~\cite[Corollary~8.8]{Wen05}). However, the native space
$\Native_G^l(\Rd)$ induced by $G$ and $\pi_{l-1}(\Rd)$ is a proper
subspace of $\NativeG$ when $d>1$. Therefore
\[
\Native_G^l(\Rd)\subsetneqq\NativeG\equiv\Hp\equiv
BL_m(\Rd),\quad{}d>1.
\]
\end{example}
%
\begin{remark}
If we have a vector distributional operator
$\vP:=\left(P_1,\cdots,P_n\right)^T$ whose distributional
Fourier transform satisfies $\vx\mapsto\norm{\FvP(\vx)}_2^2\in\pi_{2m}(\Rd)$
and
\[
\left\{a_{\alpha}D^{\alpha}:\abs{\alpha}=m,~\alpha\in\NN_0^d\right\}
\subseteq\left\{P_j:j=1,\ldots,n\right\},\quad{}\text{where
}a_{\alpha}\neq0\text{ and }m>d/2,
\]
then $\Hp\subseteq BL_m(\Rd)$. According to the Sobolev
inequality~\cite{AdaFou03}, there is a positive constant $C$ such
that $\norm{f}_{\Hp}^2\leqslant C\norm{f}_{BL_m(\Rd)}^2$ for each
$f\in\Hp$. This implies that this generalized Sobolev space $\Hp$
also has the $\Schwartz$-dense property.
\end{remark}

\begin{example}[Sobolev Splines, {\cite[Example~3]{Ye10}}]\label{e:Matern}
This is a generalization of Example~\ref{e:Sobolevspline}. Let
$\vP:=( \vQ_0^T,\cdots,\vQ_n^T )^T$, where
\[
\vQ_j:=
\begin{cases}
\left( \frac{n!\sigma^{2n-2j}}{j!(n-j)!} \right)^{1/2}\Delta^k & \text{when $j=2k$,}\\
\left( \frac{n!\sigma^{2n-2j}}{j!(n-j)!} \right)^{1/2}\Delta^k\nabla
& \text{when $j=2k+1$},
\end{cases}
\qquad k\in\NN_0,\ j=0,1,\ldots,n,\ n>d/2.
\]
Here we use $\Delta^0:=I$. We further define $L:=\vP^{\ast
T}\vP=(\sigma^2I-\Delta)^n$.

The \emph{Sobolev spline} (or Mat\'ern function) is known to be the
Green function with respect to $L$ (see
\cite[Chapter~6.1.6]{BerThAg04} and \cite[Chapter~13.2]{Fas07}), i.e.,
\[
G(\vx):=\frac{2^{1-n-d/2}}{\pi^{d/2}\Gamma(n)\sigma^{2n-d}}\left(\sigma\norm{\vx}_2\right)^{n-d/2}K_{d/2-n}
\left(\sigma\norm{\vx}_2\right),
\quad{}\vx\in\Rd,
\]
where $z\mapsto K_{\nu}(z)$ is the \emph{modified Bessel function of
the second kind of order $\nu$}. Since $\vP$ and $G$ satisfy the
conditions of Theorem~\ref{t:NS-Hp} and
$\norm{\FvP(\vx)}_2=\Theta(1)$ as $\norm{\vx}_2\to 0$, $G$ is
positive definite and the associated interpolant $s_{f,X}$ is the
same as the Sobolev spline (or Mat\'ern) interpolant.

Combining \cite[Example~3]{Ye10} and Theorem~\ref{t:NS=Hp0}, we
can determine that
\[
\Native_G^0(\Rd)\equiv\Hp\cong\Hil^n(\Rd).
\]
Moreover, this shows that the classical Sobolev space $\Hil^n(\Rd)$
becomes a reproducing-kernel Hilbert space with $\Hp$-inner product
and its reproducing kernel is given by $K(\vx,\vy):=G(\vx-\vy)$.

\end{example}

In the following example we are not able to establish that the
operator $\vP$ satisfies the conditions of Theorem~\ref{t:NS-Hp} and
so part of the connection to the theory developed in this paper is
lost. We therefore use the symbol $\Phi$ to denote the kernel
instead of $G$.
\begin{example}[Gaussians, {\cite[Example~4]{Ye10}}]\label{e:Gaussian}
The Gaussian kernel $K(\vx,\vy):=\Phi(\vx-\vy)$ derived by the
Gaussian function $\Phi$ is very important and popular in the
current research fields of scattered data approximation and machine
learning. Therefore knowledge of the
native space of the Gaussian function or the reproducing-kernel
Hilbert space of the Gaussian kernel is of significant interest. In this example we will show
that the native space of the Gaussian function is isometrically equivalent to a
generalized Sobolev space.

We firstly consider the \emph{Gaussian} function
\[
\Phi(\vx):=\frac{\sigma^d}{\pi^{d/2}}\exp(-\sigma^2\norm{\vx}_2^2),\quad{}\vx\in\Rd,\quad{}\sigma>0.
\]
We know that $\Phi$ is a positive definite function and its native
space $\Native^0_{\Phi}(\Rd)$ is a reproducing kernel Hilbert space
(see \cite[Chapter~4]{Fas07}).

Let $\vP:=( \vQ_0^T,\cdots,\vQ_n^T,\cdots )^T$, where
\[
\vQ_n:=\begin{cases}
\left( \frac{1}{n!4^n\sigma^{2n}} \right)^{1/2}\Delta^k & \text{when $n=2k$,}\\
\left( \frac{1}{n!4^n\sigma^{2n}} \right)^{1/2}\Delta^k\nabla &
\text{when $n=2k+1$,}
\end{cases}
\qquad k\in\NN_0.
\]
Here we again use $\Delta^0:=I$. Since the differential operators
are just special cases of distributional operators, the
generalized Sobolev space $\Hp$ defined by $\vP$ is the same as that derived in
\cite[Example~4]{Ye10}. Therefore we can combine Theorem~\ref{t:NS-Hp}, \ref{t:NS=Hp0} and the techniques of the proof for
\cite[Example~4]{Ye10} to obtain that
\[
\Native_{\Phi}^0(\Rd)\equiv\Hp.
\]

Moreover, it is easy to verify that $\Hp\subseteq\Hil^n(\Rd)$ for
each $n\in\NN$. According to the Sobolev embedding
theorem~\cite{AdaFou03}, we also have
$\Hp\subseteq\Real(\Cont_b^{\infty}(\Rd))$. However, $\Hp$ does not contain polynomials. If $f\in\Real(\Cont_b^{\infty}(\Rd))$ and there
is a positive constant $C$ such that
$\norm{D^{\alpha}f}_{\Linfty}\leqslant C^{\abs{\alpha}}$ for each
$\alpha\in\NN_0^d$, then $f\in\Hp$ which implies that
$f\in\Native^0_{\Phi}(\Rd)$.

If we replace the test functions space to be $\mathscr{D}$, then we
can further think of the Gaussian function $\Phi$ is a (full-space)
Green function of $L:=\exp(-\frac{1}{4\sigma^2}\Delta)$, i.e.,
$L\Phi=\delta_0$ and $\Phi,\delta_0\in\mathscr{D}'$, where
$\mathscr{D}=\Cont^{\infty}_0(\Rd)$ and its dual space
$\mathscr{D}'$ are defined in~\cite[Chapter~1.5]{AdaFou03}.

\end{example}


\section{Extensions and Future Works}\label{s:closing}

In this paper we have presented a unified theory for the generation
of conditionally positive definite functions of order $m$ as
(full-space) Green functions with respect to a distributional
operator $L:=\vP^{\ast T}\vP$ with an appropriate vector
distributional operator $\vP$. These even Green functions
$G\in\RealContinue\cap\SI$ can be used as basic functions of a
translation invariant meshfree kernel-based approximation method of
the form (\ref{interpolant})-(\ref{interpolation_conditions}). Our
analysis is limited to this translation invariant setting which does
not address the fully general situation with kernels of the form
$K(\vx,\vy)$, but is more general than the radial setting.

In Section~\ref{s:examples} we were able to show that many different
types of ``splines'' and radial basis functions can be treated with
our Green function framework. Thus, reproducing kernel Hilbert space
methods can be viewed as a natural generalization of univariate
splines (including such variations as tension splines). Other forms
of univariate splines such as smoothing splines or regression
splines can be covered using a related least squares framework, and
multivariate generalizations of these methods are widely used in
statistics and machine learning.

We only consider real-valued functions as candidates for the
generalized Sobolev spaces and Green functions in this paper, but
all the conclusions and the theorems can be extended to
complex-valued functions in a way similarly to~\cite{Wen05}. $\Hp$
may not be complete even though we extend it to complex-valued
functions. However, its completion is isometrically embedded into the tempered
distribution space $\Schwartz'$ and has the explicit form
\[
\overline{\Hp}\equiv \left\{ T\in\Schwartz':P_j
T\in\Ltwo,~j=1,\ldots,n \right\},\quad{}\textrm{if
}\vP=\left(P_1,\cdots,P_n\right)^T.
\]

The vector distributional operator $\vP$ can be further constructed
by \emph{pseudo-differential operators}. Therefore
their generalized Sobolev spaces $\Hp$ are isometrically equivalent to the
Beppo-Levi type spaces $X_{\tau}^m(\Rd)$. The paper~\cite{Bou07} shows that the radial basis function under tension may be associated
to a pseudo-differential operator in a Beppo-Levi space type. For example, if
$\vP:=(\omega_{\tau}\mathcal{F}\partial^m/\partial x_1^m, \cdots,
\left(m!/\alpha!\right)^{1/2}\omega_{\tau} \mathcal{F} D^{\alpha}, \cdots, \omega_{\tau} \mathcal{F} \partial^m/\partial
x_d^m )^T$, then
\[
\Hp\equiv X_{\tau}^m(\Rd):=\left\{f\in\RealLloc\cap\SI:\omega_{\tau}
\widehat{D^{\alpha}f}\in\Ltwo,\abs{\alpha}= m,\alpha\in\NN_0^d\right\},
\]
where $\mathcal{F}$ is a distributional Fourier transform map and $\omega_{\tau}(\vx):=\norm{\vx}_2^{\tau}$,
$0\leqslant\tau<1$. However, $\vP$ may not satisfy the condition of
Theorem~\ref{t:NS-Hp}. We have reserved these situations for our
future research.

Unfortunately, it is sometimes difficult for us to solve a Green
function matching the conditions of Theorem~\ref{t:NS-Hp} even if the
vector distributional operator $\vP$ satisfies the conditions of
Theorem~\ref{t:NS-Hp}. However, there is usually an even Green
function $G\in\Real(\Continuezero)\cap\SI$. This means that the
Green function merely has a \emph{singular point} at the origin.
According to our numerical tests of some cases, we find that this
kind of Green function can still play the role of a basic function for
the construction of a multivariate interpolant $s_{f,X}$ via
(\ref{interpolant})-(\ref{interpolation_conditions}) after some
techniques to remove the singularity. One of the numerical tests is
a two-dimensional example as below. Let $\vP:=\big(
\Delta,\sigma\nabla^T \big)^T$ with $\sigma>0$ and the Green
function with respect to $L:=\vP^{\ast
T}\vP=\Delta^2-\sigma^2\Delta$ be given by
\[
G(\vx):=-\frac{1}{2\pi\sigma^2}\left(
K_0\left(\sigma\norm{\vx}_2\right)+\log\left(\sigma\norm{\vx}_2\right)
\right),\quad{}\vx\in\RR^2,
\]
where $z\mapsto K_{\nu}(z)$ is the modified Bessel function of the
second kind of order $\nu$. We can use a transformation to
remove the singularity of $G$ as follows:
\[
G_r(\vx):=-\frac{1}{2\pi\sigma^2}\left(
K_0\left(\sigma\norm{\vx}_2+r\right)+\log\left(\sigma\norm{\vx}_2+r\right)
\right),\quad{}\vx\in\RR^2,~r>0.
\]
We guess that the interpolant via this modified Green function may
be used to approximate functions belonging to the related
generalized Sobolev space.

We merely consider the Lebesgue measure here. However, we can further
generalize our results to other measure spaces
$\left(\Omega,\mathcal{B}_{\Omega},\mu\right)$, where
$\Omega\subseteq\Rd$ and $\mathcal{B}_{\Omega}$ is the Borel set of
$\Omega$. We suppose that the bijective map
\[
\A:\left(\Omega,\mathcal{B}_{\Omega},\mu\right)\rightarrow(\Rd,\mathcal{B}_{\Rd})
\]
is differentiable at every point of $\Omega$ such that
\[
\ud\mu(\vx)=\abs{\det\left(J_{\A}(\vx)\right)}\ud\vx,\quad{}\text{where
$J_{\A}(\vx)$ is the Jacobian matrix of $\A$ at $\vx$}.
\]
According to the Radon-Nikodym Theorem~\cite{BerThAg04} it is not
difficult to gain similar conclusions when we transform the
generalized Sobolev space to be
\[
\HP^{\mu}(\Omega):=\left\{f_{\A}:=f\circ\A:f\in\Hp\right\},
\]
with the semi-inner product
\[
(f_{\A},g_{\A})_{\HP^{\mu}(\Omega)}:=\sum_{n=1}^{\infty}
\int_{\Omega}P_nf(\A(\vx))\overline{P_ng(\A(\vx))}\ud\mu(\vx),\quad{}f_{\A},g_{\A}\in\HP^{\mu}(\Omega).
\]

Finally, we do not specify any boundary conditions for the
(full-space) Green functions. Thus we may have many choices of the
Green functions with respect to the same distributional operator
$L$. In our future work we will apply a vector distributional
operator $\vP:=(P_1,\cdots,P_{n_p})^T$ and a vector boundary
operator $\vB:=(B_1,\cdots,B_{n_b})^T$ on a bounded domain $\Omega$ to
construct a reproducing kernel and its related reproducing-kernel
Hilbert space (see~\cite{FasYe10b}). We further hope to use the
distributional operator $L$ to approximate the eigenvalues and
eigenfunctions of the kernel function with the goal of obtaining
fast numerical methods to solve the interpolating
systems~(\ref{interpolant})-(\ref{interpolation_conditions}) similar
as fast multipole methods in \cite[Chapter 15]{Wen05}.


\bigskip

\end{document}